\documentclass{article}
\usepackage{graphicx}
\usepackage{amsmath,amsthm,amssymb,pifont,colortbl,amscd, wrapfig}
\usepackage{bigdelim,multirow}
\usepackage{colortbl}
\newtheorem{theorem}{Theorem}[section]
\newtheorem{lemma}[theorem]{Lemma}
\newtheorem{proposition}[theorem]{Proposition}

\newtheorem{corollary}[theorem]{Corollary}
\newtheorem{conjecture}[theorem]{Conjecture}

\theoremstyle{definition}

\newtheorem{remark}[theorem]{Remark}
\newtheorem*{acknowledgments}{Acknowledgments}

\def\ev{\mathrm{ev}}

\def\mod{\mathop{\mathrm{mod}}\nolimits}

\def\ev{\mathop{\mathrm{ev}}\nolimits}

\def\Span{\mathop{\mathrm{Span}}\nolimits}
\def\Span{\mathop{\mathrm{Span}}\nolimits}

\def\co{\colon\thinspace}

\newcommand{\Z}{\mathbb{Z}[q,q^{-1}]}
\newcommand{\s}[3]{x_{#1,#2}^{(#3)}}
\newcommand{\tPhi}{\tilde \Phi}
\begin{document}
\title{Bing doubling and the colored Jones polynomial}
\author{Sakie Suzuki\thanks{Faculty of Mathematics, Kyushu University, Fukuoka, 819-0395, Japan. E-mail address: \texttt{sakie@math.kyushu-u.ac.jp}} }
\date{June 18, 2013}
\maketitle

\begin{center}
\textbf{Abstract}
\end{center}
Bing doubling is an operation which gives a satellite of a knot.
It is also applied  to a  link by specifying a component of the link.
We give a formula to compute the reduced colored Jones polynomial of a Bing double by using that of the companion.
This formula enables us to compute a lot of examples of  the reduced colored Jones polynomial of Bing doubles.
Moreover,  from this formula we can derive a divisibility property of the unified Witten-Reshetikhin-Turaev invariant of integral homology spheres obtained by
$\pm 1$-surgery along Bing doubles of knots.
This result is applied to   the Witten-Reshetikhin-Turaev invariant and the Ohtsuki series of these integral homology spheres.
\section{Introduction}

Bing doubling \cite{Bing} is an operation which gives the satellite $B(K)$ of a framed  knot $K$ as in Figure \ref{fig:bing},
i.e., $B(K)$ is the $2$-component link obtained from $K$ by duplicating along its framing and making the clasps on it.

\begin{figure}
\centering
\begin{picture}(200,40)
\put(10,0){\includegraphics[width=7.3cm,clip]{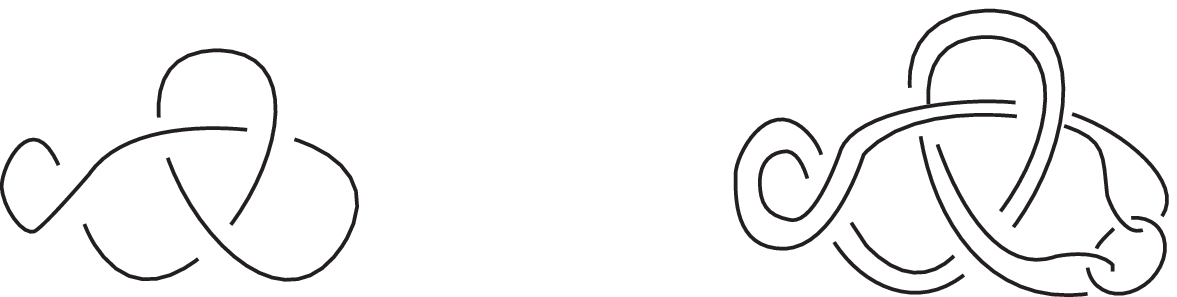}}
\put(-20,20 ){$K=$}
\put(93,20){$B(K)=$}
\end{picture}
\caption{The trefoil knot  $K$ with $4$-framing and its Bing double $B(K)$  }\label{fig:bing}
\end{figure}%

Bing doubling has been  studied  in the context of  link concordance \cite{Ci, Ch1, Ch2, Ch3, Co, Fr2,H}.
Bing doubling  is also important in the study of Milnor's $\bar \mu$ invariants \cite{Co1} and  finite type invariants of knots \cite{H3}.
Some of  classical invariants such as the multivariable Alexander polynomial and  the Arf invariant are useless in the study of Bing doubles since those cannot distinguish iterated Bing doubles from unlink.
Cimasoni \cite{Ci} used Rasmussen invariant  and Cha, Livingston and Ruberman \cite{Ch3} used the Ozv\'ath-Szab\'o invariant  and the Manolescu-Owens invariant  to study Bing doubles.  Cochran and Melvin \cite{mel} used the quantum  $SO(3)$ invariant  \cite{KM} to study the Milnor degree of $3$-manifolds which are obtained from $S^3$ by
surgery along  iterated Bing doubles of the Hopf link.

Our interest here is in the relationship between Bing doubling and quantum invariants, especially
\textit {the reduced colored Jones polynomial} of links and  \textit{the unified Witten Reshetikhin Turaev (WRT) invariant} of integral homology spheres. 

The colored Jones polynomial $J_{L; W_1, \ldots, W_n} \in \mathbb{Z}[q^{1/4}, q^{-1/4}]$  is one of  quantum invariants, which is defined  for a  framed link $L=L_1\cup \cdots \cup L_n$ with the $i$th component   colored by a finite dimensional representation $W_i$ of the quantized enveloping algebra
$U_h(sl_2)$ of the Lie algebra $sl_2$.
 
Habiro \cite{H2} defined the reduced colored Jones polynomial
$J_{L; P'_{l_1}, \ldots,  P'_{l_n}}\in \mathbb{Q}(q^{\frac{1}{2}})$ for a  framed link $L=L_1\cup \cdots \cup L_n$ with the $i$th component   colored by 
 an element $P'_{l_i}$ in the representation ring of $U_h(sl_2)$ over $\mathbb{Q}(q^{\frac{1}{2}})$.
By using the  reduced colored Jones polynomial, he constructed  the unified WRT invariant $J_M\in \widehat{\mathbb{Z}[q]}$
  of an integral homology sphere $M$, where $\widehat{\mathbb{Z}[q]}$ is  so-called \textit{Habiro ring}. 
For each root of unity $\zeta$, we obtain the $SU(2)$  WRT invariant $\tau_{\zeta}(M)\in \mathbb{Z}[\zeta]$ from $J_M$ by evaluating $q=\zeta$.
Here, for integral homology spheres,  the $SU(2)$  WRT invariant is equal to the quantum $SO(3)$ invariant which is defined for each root of unity of odd order $\geq 3$.
We can obtain also    \textit{the  Ohtsuki series} $\tau^O(M)\in \mathbb{Z}[[q-1]]$ from $J_M$  by taking the  power series expansion at $q=1$.
See Sections  \ref{un} and  \ref{uni} for the definitions of the reduced colored Jones polynomial and the unified WRT invariant.

In  \cite{H2}, Habiro proved that  $J_{L;  \tilde P'_{l_1},  \ldots,  \tilde P'_{l_n}}$ for an  $n$-component, algebraically-split, $0$-framed link is contained in a certain ideal $Z^{(l_1,\ldots, l_n)}_a\subset \Z$, where $\tilde P'_{l}$ is a certain normalization of $P'_{l}$.
In  \cite{sakie1, sakie2, sakie3}, the author improved his result in the cases of  \textit{Brunnian} links, \textit{ribbon} links, and  \textit{boundary} links, i.e.,
we proved that  $J_{L;  \tilde P'_{l_1},  \ldots,  \tilde P'_{l_n}}$ for  an $n$-component, $0$-framed
 Brunnian link is contained in an ideal $\tilde Z^{(l_1,\ldots, l_n)}_{Br}\subset Z^{(l_1,\ldots, l_n)}_a$,  and $J_{L;  \tilde P'_{l_1},  \ldots,  \tilde P'_{l_n}}$  for an $n$-component, $0$-framed  ribbon  or boundary link  is contained in an ideal $Z^{(l_1,\ldots, l_n)}_{r,b}\subset \tilde Z^{(l_1,\ldots, l_n)}_{Br}$. 
 These results were proved by using \textit{the universal $sl_2$ invariant of bottom tangles}.
In \cite{sakie4}, the author also proved that the ideals $\tilde Z^{(l_1,\ldots, l_n)}_{Br}$ and $Z^{(l_1,\ldots, l_n)}_{r,b}$ are principal, and gave the generators of these ideals.
(In \cite{sakie3, sakie4}, there appears another ideal $Z^{(l_1,\ldots, l_n)}_{Br}\subset \Z$ such that $\tilde Z^{(l_1,\ldots, l_n)}_{Br}=Z^{(l_1,\ldots, l_n)}_{Br}\cup Z^{(l_1,\ldots, l_n)}_a$.)

For a framed link $L=L_1\cup \cdots \cup L_n$, we denote by $B(L; 1)$ the $(n+1)$-component link obtained from $L$ by applying  Bing doubling to $L_1$.
It is not difficult to see that the reduced colored Jones polynomial of  $B(L; 1)$ is a linear combination of those of $L$ (Lemma \ref{10}),
i.e., we have 
\begin{align*}
J_{B(L; 1); P'_i, P'_j, P'_{l_2},\ldots,  P'_{l_n}}=\sum_{l\geq 0}x_{{i,}{j}}^{({l})}J_{L; P'_l, P'_{l_2},\ldots,  P'_{l_n}}
\end{align*}
 for certain 
$x_{{i,}{j}}^{({l})}\in \mathbb{Q}(q^{\frac{1}{2}})$.

One of two main results in this paper is Theorem \ref{1}, which gives the coefficients   $x_{{i,}{j}}^{({l})}$ explicitly.
This result gives many examples of computations of the reduced colored Jones polynomials of  Bing doubles.

The other main result  is Theorem \ref{3.1}, which says a divisibility property of the difference $J_{M(B(K);  \epsilon, \epsilon'))}-J_M$ of unified WRT invariants,
where $M$ is  an integral homology sphere and $K$ is a $0$-framed knot in $M$, and $M(B(K);  \epsilon, \epsilon')$ is the  integral homology sphere obtained from $M$ by surgery along the Bing double $B(K)$ 
with $ \epsilon, \epsilon'\in \{\pm 1\}$ framings.
This result is applied to prove some improvements of Habiro's results in \cite{H2} for the WRT invariant and the Ohtsuki series of integral homology spheres.

The rest of the paper is organized as follows.
In Section \ref{pre}, we recall the definitions of the reduced colored Jones polynomial and the unified WRT invariant.
In Section \ref{results}, we give the main results.
In Section \ref{Ex},  we study  the coefficients  $x_{{i,}{j}}^{({l})}\in \mathbb{Q}(q^{\frac{1}{2}})$ further.
In Section \ref{Ex2},    we compute some of  the reduced colored Jones polynomials of Milnor's link as in Figure \ref{fig:brunnianex}, which is obtained from the Borromean rings by applying Bing doubling repeatedly.
Section \ref{pr} is devoted to the proofs.

\section{Preliminaries}\label{pre}

In this section, we recall from \cite{H2}  the definition of the reduced  colored Jones polynomial and the unified WRT invariant.
 
We use the following $q$-integer  notations:
\begin{align*}
&\{i\} = q^{i/2}-q^{-i/2},\quad  \{i\}_{n} = \{i\}\{i-1\}\cdots \{i-n+1\},
\\
& \{n\}! = \{n\}_{n},\quad  \begin{bmatrix} i \\ n \end{bmatrix}   = \{i\}_{n}/\{n\}!,
\end{align*}
for $i\in \mathbb{Z}, n\geq 0$.

\subsection{Reduced Colored Jones polynomial}\label{un}
The colored Jones polynomial $J_{L; W_1, \ldots, W_n}\in \mathbb{Z}[q^{1/4},q^{-1/4}]$ is defined  for an $n$-component  framed link $L$ with  the $i$th component $L_i$  colored by a finite dimensional representation $W_i$ of the quantized enveloping algebra
$U_h(sl_2)$ of the Lie algebra $sl_2$. 
In this paper, we follow  \cite{H2} for the definition of  colored Jones polynomial.
The reduced colored Jones polynomial is defined as a linear combination of the colored Jones polynomial as follows.

For $m\geq 0$, let  $V_{m}$ denote the $(m+1)$-dimensional irreducible representation of $U_h(sl_2)$.
Let $\mathcal{R}$  denote the representation ring  of  $U_h(sl_2)$ over  $\mathbb{Q}(q^{\frac{1}{2}})$, i.e.,
$\mathcal{R}$ is the $\mathbb{Q}(q^{\frac{1}{2}})$-algebra 
\begin{align*}
\mathcal{R}= \Span _{\mathbb{Q}(q^{\frac{1}{2}})}\{V_m \  | \ m\geq 0\}
\end{align*}
with the multiplication induced by the tensor product.

We define the  colored Jones polynomial of  an $n$-component  framed link $L$ with the $i$th component $L_i$ colored by $X_i=\sum_{l_i\geq 1}a_{l_i}^{(i)}V_{l_i}\in \mathcal{R}$ by
 \begin{align*}
J_{L; X_1,\ldots , X_n}=\sum_{l_1, \ldots, l_n \geq 1}a_{l_1}^{(1)}\cdots a_{l_n}^{(n)}J_{L; V_{l_1},\ldots , V_{l_n}} \in \mathbb{Q}(q^{\frac{1}{2}}).
 \end{align*}
 
 For $l\geq 0$, set 
\begin{align*}
P'_l&=\frac{1}{\{l\}!} \prod _{i=0}^{l-1}(V_1-q^{i+\frac{1}{2}}-q^{-i-\frac{1}{2}}) \in \mathcal{R}.
\end{align*}

For  an $n$-component  framed link $L$, we call $J_{L; P'_{l_1},\ldots , P'_{l_n}}$ \textit{the reduced colored Jones polynomial} of $L$.
\subsection{Unified WRT invariant}\label{uni}

For $k\geq 0$, set
\begin{align*}
\mathcal{P}_k&=\Span_{\Z}\{q^{-\frac{1}{4}l(l-1)} P'_l\ | \   l\geq k\}  \subset \mathcal{R},
\end{align*}
Set 
\begin{align*}
\hat {\mathcal{P}}&=\varprojlim_{k\geq 0}  \mathcal{P}_0/\mathcal{P}_k.
\end{align*}

Set 
\begin{align*}
\omega ^{\pm 1}=\sum_{l=0}^{\infty} (\pm 1)^l q^{\pm \frac{1}{4}l(l+3)}P'_l \in \hat {\mathcal{P}}.
\end{align*}
Let $\widehat{\mathbb{Z}[q]}$ be the Habiro ring, i.e., 
\begin{align*}
\widehat{\mathbb{Z}[q]}&=\varprojlim_{n\geq 0} \mathbb{Z}[q]/((1-q)(1-q^2)\cdots(1-q^n)).
\end{align*}

Let $M$ be the integral homology sphere obtained by surgery along  an algebraically-split  link $L=L_1\cup \cdots \cup L_n$ in $S^3$ with framings $\epsilon_1,\ldots, \epsilon_n \in \{\pm 1\}$.
Habiro \cite{H2}  constructed  the unified WRT invariant $J_M\in \widehat{\mathbb{Z}[q]}$ of $M$ by
\begin{align*}
J_M&=J_{L^0; \omega ^{-\epsilon_1},\ldots,\omega^{-\epsilon_n}} 
\\
&:=\sum_{l_1,\ldots,l_n =0}^{\infty}\Big( \prod _{i= 1,\ldots,n} (-\epsilon_i )^{l_i} q^{-\epsilon_i\frac{1}{4}{l_i}({l_i}+3)}\Big)
J_{L^0; P_{l_1}',\ldots,P_{l_n}'} 
\in \widehat{\mathbb{Z}[q]},
\end{align*}
where $L^0$ is the link obtained from $L$ by changing all framings to $0$.

For  an integral homology sphere $M$ and  a root of unity $\zeta$, let $\tau_{\zeta}(M)\in \mathbb{Z}[\zeta]$ be the WRT invariant \cite{Re2} at $\zeta$.
 Let $\ev_{\zeta}  \co \widehat{\mathbb{Z}[q]} \rightarrow \mathbb{Z}[\zeta]$ be the evaluation map.
 For each root of unity, the WRT invariant  is obtained from the unified WRT invariant   as follows.
\begin{theorem}[Habiro \cite{H2}]
Let $M$ be an integral homology sphere. For each root of unity $\zeta$, 
we have
\begin{align*}
\ev_{\zeta}(J_M)=\tau_{\zeta}(M).
\end{align*}
\end{theorem}

For  an integral homology sphere $M$, let  $\tau^{O}(M)$ denote the Ohtsuki series  \cite{O}
and  $\imath(J_M)\in \mathbb{Z}[[q-1]]$ denote the power series expansion of $J_M$ at $q=1$.

\begin{theorem}[Habiro \cite{H2}]
For an integral homology sphere $M$, 
we have
\begin{align*}
\imath(J_M)= \tau^{O}(M).
\end{align*}
\end{theorem}
\section{Main results}\label{results}
In this section, we give the main results of this paper for the  reduced colored Jones polynomial (Theorem \ref{1}) and  the unified WRT invariant (Theorem \ref{3.1}). 
We give also applications  to the WRT invariant (Proposition \ref{wrts2}) and the Ohtsuki series (Proposition \ref{oht}).
\subsection{Result for the reduced colored Jones polynomial}\label{colo}
It is known that the colored Jones polynomial of a satellite is a linear combination of the colored Jones polynomials of the companion \cite{Re, MS}.
The following lemma says that a similar situation works for the reduced colored Jones polynomials of Bing doubles.

\begin{lemma}\label{10}
There exists $x_{{i,}{j}}^{({l})}\in \mathbb{Q}(q^{1/2})$, $i,j,l\geq 0$, such that
\begin{align*}
J_{B(L; 1); P'_i, P'_j, W_1,\ldots,  W_{n-1}}=\sum_{l\geq 0}x_{{i,}{j}}^{({l})}J_{L; P'_l, W_1,\ldots,  W_{n-1}},
\end{align*}
for any $n$-component framed  link $L=L_1\cup \cdots \cup L_n$   and  $W_1,\ldots, W_{n-1}\in \mathcal{R}$.
\end{lemma}
\begin{proof}
For simplicity, we prove the claim for $n=1$, i.e.,
$L=K$ is a framed  knot.  We can prove the assertion for $n\geq 2$ similarly.
Since  $\{V_l\}_{l\geq 0}$ is a basis of $\mathcal {R}$,   the colored Jones polynomial  $J_{B(K); P'_i, P'_j} $  is 
a linear sum of  colored Jones polynomials $J_{B(K); V_t,V_u}$ in $\mathcal{R}$.
Since the colored  Jones polynomial $J_{B(K); V_t,V_u}$ is a linear sum of colored Jones polynomials $J_{K; V_s}$  in $\mathbb{Z}[q^{1/2},q^{-1/2}]$ (cf. \cite[Theorem 3.1]{MS}),
and since $\{P'_l\}_{l\geq 0}$ is also a basis of $\mathcal {R}$,
$J_{B(K); V_t,V_u}$  is a linear sum of colored Jones polynomials $J_{K; P'_l}$ in  $\mathcal {R}$. 
Consequently, $J_{B(K); P'_i, P'_j} $  is a linear sum of $J_{K; P'_l}$ in  $\mathcal {R}$.
Moreover, in each step, the coefficients of the linear sum do not depend on the knot $K$.
Hence we have the assertion.
\end{proof}

One of our main results is the  following, which we prove  in Section \ref{proof}.

\begin{theorem}\label{1}
For $i,j, l \geq 0$, we have
\begin{align}\label{s}
\s{i}{j}{l}&= \delta_{i,j }(-1)^{i}\{l\}!\alpha_{i,l},
\end{align}
where
\begin{align*}
\alpha_{i,l}&=\sum_{k=0}^i(-1)^{k}\begin{bmatrix}2i+1 \\ k\end{bmatrix}\begin{bmatrix}2i+l-2k+1 \\ 2l+1\end{bmatrix}.
\end{align*}
\end{theorem} 

Theorem \ref{1} enables us to compute many examples of the reduced  colored Jones polynomial of Bing doubles.
Especially in Section \ref{Ex2}, we give examples with Milnor's links, which are obtained from the Borromean rings by applying Bing doubling repeatedly.
Moreover, from Theorem \ref{1} we can derive a divisibility  property of the unified WRT invariant as in the following section.

\subsection{Result for the unified WRT invariant}
%We give another main result in this paper, which will be applied to the WRT invariant and the Ohtsuki series in the following sections.

For $m\geq 1$, let $\Phi _m=\prod _{d|m}(q^d-1)^{\mu (\frac{m}{d})}\in \mathbb{Z}[q]$ denote the $m$th cyclotomic polynomial, where $\prod _{d|m}$ denotes the
product  over all the positive divisors $d$ of $m$, and  $\mu$ is the M\"obius function.

Let $M$ be an integral homology sphere, and  $L=L_1 \cup \cdots \cup L_n$ an algebraically-split   link  in $M$.
For  $\epsilon_1,\ldots,\epsilon_{n}\in \{\pm 1\}$, let $M(L; \epsilon_1,\ldots,\epsilon_n)$ denote the integral homology sphere obtained from $M$ by surgery along  $L$ with  framings  $\epsilon_1,\ldots,\epsilon_n$.

In what follows, two integral homology spheres $M$ and $M'$ are said to be related by a \textit{special Bing double surgery} if there is a $0$-framed knot $K$ in $M$ and signs
$\epsilon, \epsilon' \in \{\pm 1\}$ such that  $M(B(K); \epsilon, \epsilon')$ is orientation-preserving homeomorphic to $M'$.

\begin{remark}
Note that $M(B(K); \epsilon, \epsilon')=M(W_{\epsilon'}(K), \epsilon)$, where $W_{\epsilon'}(K)$ is the Whitehead double of $K$ with a clasp of $\epsilon'$-type.
\end{remark}

The other main result  in this paper is the following, which we prove in Section \ref{pr4}.

\begin{theorem}\label{3.1}
Let $M$ and $M'$ be integral homology spheres related by a special Bing double surgery. 
Then we have 
\begin{align}\label{th3}
J_{M'}-J_{M} \in \Phi_1^2\Phi_2^2\Phi_3\Phi_4\Phi_6\widehat{\mathbb{Z}[q]}=(q^4-1)(q^6-1)\widehat{\mathbb{Z}[q]}.
\end{align}
In particular, if   $M=S^3$,  then we have
\begin{align*}  
J_{M'}-1 \in \Phi_1^2\Phi_2^2\Phi_3\Phi_4\Phi_6\widehat{\mathbb{Z}[q]}.
\end{align*}
\end{theorem}
\begin{remark}
Note that Theorem \ref{3.1} implies  (\ref{th3}) for integral homology spheres $M$ and $M'$ related by a \textit{sequence} of special Bing double surgeries, i.e.,
 when   there is a sequence $M=M_1, M_2,\ldots, M_l=M'$ of  integral homology spheres such that for each $i=1,\ldots, l-1$, $M_i$ and $M_{i+1}$ are related by a 
special Bing double surgery.
\end{remark}
\begin{remark}
Habiro's result  \cite[Proposition 12.15]{H2}  implies that for integral homology spheres $M$ and $M'$ with the same Casson invariant, we have
\begin{align}\label{wrth}
J_{M'}-J_M\in \Phi_1^2\Phi_2\Phi_3\Phi_4\Phi_6\widehat{\mathbb{Z}[q]}=(q-1)(q^2+1)(q^6-1)\widehat{\mathbb{Z}[q]}.
\end{align}
Since $M'$ and $M$ in Theorem \ref{3.1}  has the same Casson invariant,
Theorem \ref{3.1} is an improvement of his result with respect to $\Phi_2$.
\end{remark}

\begin{remark}\label{melv}
In  \cite{mel2},  Cochran and Melvin introduced  \textit{the quantum p-order} ${\mathfrak{o}}_p(M)$ for a prime $p$ and a $3$-manifold $M$, 
which is  \textit{the p-order} of the quantum $SO(3)$ invariant $\tau_{\zeta_{p}}^{SO(3)}(M)\in \mathbb{Q}(\zeta_p)$ associated with  a primitive $p$th root of unity $\zeta_p$.
Here, the $p$-order of an element $g$ in the cyclotomic field $\mathbb{Q}(\zeta_p)$ is the exponent of the prime ideal $(\zeta_p-1)\subset \mathbb{Z}[\zeta_p]$
in the prime decomposition of   the fractional ideal generated by $g$.

 For prime $p>3$, they  \cite[Section 3]{mel}
 proved that \begin{align*}
\hat {\mathfrak{o}}_p (S^3{(B(L;1)}; 0,0,f_2,\ldots, f_n))=\hat {\mathfrak{o}}_p(S^3({L}; 0, f_2,\ldots,f_n))+1
\end{align*}
for an $n$-component link $L$ and  integers $f_2,\ldots, f_n$, where   $\hat{\mathfrak{o}}_p={\mathfrak{o}}_p\cdot \frac{2}{(p-3)}$ is a rescaling version of ${\mathfrak{o}}_p$.
Though Theorem \ref{3.1} appears to say nothing about the quantum $p$-order of $S^3(B(K); \epsilon, \epsilon')$, 
Theorem \ref{1} could be applied to study  the quantum  $p$-order of $3$-manifolds obtained from  $S^3$ by surgery along Bing doubles with arbitrary framings.
 \end{remark}

\begin{figure}
\centering
\begin{picture}(50,20)
\put(10,0){\includegraphics[width=2cm,clip]{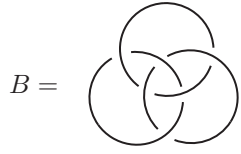}}
\put(-20,20){$B=$}
\end{picture}
\caption{Borromean rings $B$  }\label{fig:borro1}
\end{figure}%

There is a family of examples of integral homology spheres which do not have the divisibility property in Theorem \ref{3.1} as follows. 

Let $B$  be the Borromean rings  depicted in Figure \ref{fig:borro1} and
 $B_1, B_2, B_3$  the components of $B$.
For $i,j,k \in \mathbb{Z}$, let $M_{i,j,k}$ be the integral homology sphere obtained from $S^3$ by surgery along $B_1, B_2, B_3$ with framings $-1/i$, $-1/j$, $-1/k$,
respectively.

We prove the following proposition in Section \ref{pr3}.
\begin{proposition}\label{p3}
For $i,j,k \in \mathbb{Z}$, we have
\begin{align*}
J_{M_{i,j,k}}-1&\equiv 6ijk\Phi_2 \quad   (\mod \Phi_2^2). 
\end{align*}
\end{proposition}
Note that the Casson invariant of $M_{i,j,k}$ is $-6ijk$. It appears to be natural to raise  the following  conjecture.

 \begin{conjecture}
For integral homology spheres $M$ and $M'$ with the same Casson invariant,  we have 
\begin{align*}
J_{M'}-J_M \in \Phi_1^2\Phi_2^2\Phi_3\Phi_4\Phi_6\widehat{\mathbb{Z}[q]}.
\end{align*}
\end{conjecture}

\subsection{Applications to the  WRT invariant}

From Theorem \ref{3.1} by substituting a root of unity $\zeta$ for $q$, we obtain
the following result for the WRT invariant.
\begin{proposition}\label{wrts2}
Let $M$ and $M'$ be integral homology spheres related by a special Bing double surgery. 
For any root of unity $\zeta$, we have
\begin{align*}
\tau_{\zeta}(M')-\tau_{\zeta}(M)\in (\zeta^4-1) (\zeta^6-1)\mathbb{Z}[\zeta].
\end{align*}
\end{proposition}
Proposition \ref{wrts2} improves the result obtained from Habiro's result (\ref{wrth}) by substituting $\zeta$  for $q$,  in the case of  $\zeta=\zeta_{2p^e}$ as follows.
Here, $\zeta_{2p^e}$ denotes a primitive $2p^e$th root of unity   for a prime $p$ and an integer  $e\geq 0$.

\begin{corollary}
Let $M$ and $M'$ be integral homology spheres related by a special Bing double surgery. 
\begin{itemize}
\item[\rm{(i)}]
For $e\geq 0$, we have
\begin{align*}
\tau_{\zeta_{2^e}}(M')-\tau_{\zeta_{2^e}}(M)\in (\zeta_{2^e}^4-1) (\zeta_{2^e}^2-1) \mathbb{Z}[\zeta_{2^e}].
\end{align*}
\item[\rm{(ii)}]
 For   $e \geq 0$, we have
\begin{align*}
\tau_{\zeta_{2\cdot 3^e}}(M')-\tau_{\zeta_{2\cdot 3^e}}(M)\in (\zeta_{2\cdot 3^e}+1)^2(\zeta_{2\cdot 3^e}^2-\zeta_{2\cdot 3^e}+1) \mathbb{Z}[\zeta_{2\cdot 3^e}].
\end{align*}
\item[\rm{(iii)}]
For a  prime $p>3$ and $e \geq 0$, we have
\begin{align*}
\tau_{\zeta_{2p^e}}(M')-\tau_{\zeta_{2p^e}}(M)\in (\zeta_{2p^e}+1)^2 \mathbb{Z}[\zeta_{2p^e}].
\end{align*}

\end{itemize}
\end{corollary}
\subsection{Applications to the  Ohtsuki series}

Theorem \ref{3.1} can also be applied to the Ohtsuki series as follows.

For an integral homology sphere $M$, we write
\begin{align*}
\tau^O (M)=\imath(J_M)=1+\sum_{i\geq 1}\lambda_i(M)\hbar^i\in \mathbb{Z}[[\hbar]],
\end{align*}
where we set  $\hbar=(q-1)$.

For integral homology spheres $M, M'$,  set 
\begin{align*}
\Lambda_k(M, M')=\lambda_{k}(M')-\lambda_{k}(M)
\end{align*}
for $k\geq 1$. Note that
$\Lambda_k(S^3, M')=\lambda_{k}(M').$

For integral homology spheres $M$ and $M'$ with the same Casson invariant, Habiro's result  \cite[Proposition 12.28]{H2}  implies
\begin{align*}
\sum_{i=0}^k b_i \Lambda_{k-i+2}(M, M')\equiv 0 \  (\mod\mathbb{Z})
\end{align*}
for $k\geq 0$ and  a certain series $b_i  \in \mathbb{Z}[\frac{1}{6}]$ with $b_0=1/12$.

We can improve his result as follows.

\begin{proposition}\label{oht}
Let $M$ and $M'$ be integral homology spheres related by a special Bing double surgery. 
%In the situation of Theorem \ref{3.1}, 
For $k\geq 0$, we have
\begin{align}\label{oht2}
\sum_{i=0}^k c_i \Lambda_{k-i+2}(M, M')\equiv 0 \  (\mod\mathbb{Z}),
\end{align}
where $c_0=1/24,  c_1=-1/8, c_2=59/288, c_3=-17/72, \ldots \in \mathbb{Z}[\frac{1}{6}]$ are determined by
\begin{align*}
\sum_{i\geq 0}c_i \hbar^i =&\frac{1}{(q+1)^2(q^2+q+1)(q^2+1)}
\\
=&\frac{1}{24 + 72 \hbar+ 98 \hbar^2 + 76 \hbar^3 + 35 \hbar^4 + 9 \hbar^5 + \hbar^6}.
\end{align*}
In particular, for $k\geq 2$, $\Lambda_{k+2}(M, M')$ $(\mod 24)$ is determined by $\Lambda_{i}(M, M')$ for $i=2,\ldots, k$. 
\end{proposition}
\begin{proof}
The proof is similar to that of  \cite[Proposition 12.16]{H2}  if we use  Theorem \ref{3.1} instead of  \cite[Proposition 12.14]{H2}
and $\Lambda_i$ instead of $\lambda_i$.
The latter part follows from  $c_0=1/24$ and $c_2=-1/8 \in (1/24)\mathbb{Z}$.
\end{proof}
\begin{remark}
In \cite[Corollary 4.1]{takata}, Takata determined the image of $\lambda_k$, $k=1,\ldots,6$,  for integral homology spheres.
The congruence relations (\ref{oht2}) for $M=S^3$, $k=1,\ldots, 6$, can be derived also from her result.
Note that the relation (v) in \cite[Corollary 4.1]{takata} contains an error which  can be corrected by using \cite[Theorem 3.1]{takata}.
\end{remark}

\section{Properties of the coefficients $\s{i}{j}{l}$ in Theorem \ref{1}.}\label{Ex}
In this section, we  study the coefficients $\s{i}{j}{l}=\delta_{i,j }(-1)^{i}\{l\}!\alpha_{i,l}$  which appears in Theorem \ref{1}.

\subsection{Symmetry property}

The element $\alpha_{m,n}$ has a symmetry property as follows, which we prove in Section \ref{proof}.
\begin{lemma}\label{l0}
For $m,n\geq 0$,  we have 
\begin{align*}
\alpha_{m,n}=\frac{\{2m+1\}!}{\{2n+1\}!}\alpha_{n,m}.
\end{align*}
\end{lemma}
\begin{corollary}\label{c0}
For $m,n\geq 0$,  we have 
\begin{align*}
\s{m}{m}{n}=(-1)^{m+n}\frac{\{2m+1\}!\{n\}!}{\{2n+1\}!\{m\}!}\s{n}{n}{m}.
\end{align*}
\end{corollary}
\begin{proof}
By (\ref{s}) and Lemma \ref{l0}, we have 
\begin{align*}
\s{m}{m}{n}&=(-1)^{m}\{n\}!\alpha_{m,n}
\\
&=(-1)^{m}\{n\}!\frac{\{2m+1\}!}{\{2n+1\}!}\alpha_{n,m}
\\
&=(-1)^{m}\{n\}!\frac{\{2m+1\}!}{\{2n+1\}!}\Big((-1)^{n}\frac{1}{\{m\}!}\s{n}{n}{m}\Big)
\\
&=(-1)^{m+n}\frac{\{2m+1\}!\{n\}!}{\{2n+1\}!\{m\}!}\s{n}{n}{m}.
\end{align*}
Hence we have the assertion.
\end{proof}
\subsection{Particular values}

For $m,n\geq 0$, note that
\begin{align}
\alpha_{m,n}=\sum_{k=0}^{\lfloor  m-\frac{n}{2} \rfloor}(-1)^{k}\begin{bmatrix}2m+1 \\ k\end{bmatrix}\begin{bmatrix}2m+n-2k+1 \\ 2n+1\end{bmatrix},
\label{lambda2}
\end{align}
where the upper bound comes from 
\begin{align*}
\begin{bmatrix}2m+n-2k+1\\ 2n+1\end{bmatrix} =0
\end{align*}
for $2m+n-2k+1<2n+1$, i.e., for $k>m-n/2$.

We can compute particular values of $\alpha_{m,n}$ as follows.

\begin{proposition}\label{p1}
We have 
\begin{itemize}
\item[\rm{(i)}] 
$\alpha_{m,n}=0$   unless  $ \frac{n}{2}   \leq m \leq 2n$,
\item[\rm{(ii)}]  
$\alpha_{m,2m}=1$  for $m\geq 0$, and 
\item[\rm{(iii)}] 
$\alpha_{m,2m-1}=\{4m\}/\{1\}$ for $m\geq 0$.
\end{itemize}
\end{proposition}
\begin{proof}
The assertion  (i) follows from  (\ref{lambda2}) and Lemma \ref{l0}.

The assertion  (ii)  follows from 
\begin{align*}
\alpha_{m,2m}&=\sum_{k=0}^0(-1)^{k}\begin{bmatrix}2m+1 \\ k\end{bmatrix}\begin{bmatrix}4m-2k+1 \\ 4m+1\end{bmatrix}
\\
&=\begin{bmatrix}2m+1 \\ 0\end{bmatrix}\begin{bmatrix}4m+1 \\ 4m+1\end{bmatrix}
\\
&=1.
\end{align*}

The assertion (iii) follows from 
\begin{align*}
\alpha_{m,2m-1}&=\sum_{k=0}^0(-1)^{k}\begin{bmatrix}2m+1 \\ k\end{bmatrix}\begin{bmatrix}4m-2k \\ 4m-1\end{bmatrix}
\\
&=\begin{bmatrix}2m+1 \\ 0\end{bmatrix}\begin{bmatrix}4m \\ 4m-1\end{bmatrix}
\\
&=\begin{bmatrix}4m \\ 1\end{bmatrix}
\\
&=\{4m\}/\{1\}.
\end{align*}
This completes the proof.
\end{proof}
Proposition \ref{p1} implies the following.
\begin{corollary}\label{c1}
We have
\begin{itemize}
\item[\rm{(i)}] 
$\s{m}{m}{n}=0$ unless  $ \frac{n}{2}   \leq m \leq 2n$,
\item[\rm{(ii)}]  
$\s{m}{m}{2m}=(-1)^m\{2m\}!$ for $m\geq 0$, and 
\item[\rm{(iii)}] 
$\s{m}{m}{2m-1}=(-1)^m\{2m-1\}!\{4m\}/\{1\}$ for $m\geq 0$.
\end{itemize}
\end{corollary}

\subsection{Divisibility property with respect to the cyclotomic polynomials}

Let us study the divisibility property  of $\alpha_{m,n}$ and  $\s{n}{n}{m}$  with respect to the cyclotomic polynomials.
In what follows, we use also  the symmetric version of the cyclotomic polynomial $\tilde \Phi _l= \prod _{d|l}(q^{d/2}-q^{-d/2})^{\mu (\frac{l}{d})}\in \mathbb{Z}[q^{1/2},q^{-1/2}]$ for $l\geq 1$.
For $f\in \mathbb{Z}[q^{1/2},q^{-1/2}]$, $f\neq 0$, let $d_l\big(f\big)$ be the largest integer $i$ such that $f\in \tilde \Phi _l^i\mathbb{Z}[q^{1/2},q^{-1/2}]$.
\begin{lemma}\label{ll1}
For $l\geq 1$ and $n\geq 0,$ we have
\begin{align*}
d_l(\{n\}!)=\lfloor \frac{n}{l} \rfloor.
\end{align*}
\end{lemma}
\begin{proof}
The assertion follows from
\begin{align*}
d_l\big( \{i\}\big)=d_l(q^{i/2}-q^{-i/2})=
\begin{cases}1\quad \text{if }l|i,
\\
0\quad \text{otherwise},
\end{cases}
\end{align*}
for $i\geq 0$.
\end{proof}

Lemma \ref{ll1} implies the following result.

\begin{corollary}\label{huu}
For $l\geq 1$ and  $m,n\geq 0$,  we have
\begin{align*}
d_l(\s{m}{m}{n})= \lfloor \frac{n}{l} \rfloor+d_l(\alpha_{m,n}).
\end{align*}
\end{corollary}

Lemmas \ref{l0} and \ref{ll1} imply the following result.

\begin{corollary}\label{cc0}
For $l\geq 1$ and   $m,n \geq 0$, we have
\begin{align*}
d_l(\alpha_{m,n})= \lfloor \frac{2m+1}{l} \rfloor-\lfloor \frac{2n+1}{l} \rfloor+d_l(\alpha_{n,m}).
\end{align*}
\end{corollary}
In view of Corollaries \ref{huu} and \ref{cc0}, in what follows, we study $d_l(\alpha_{m,n})$ for $0\leq m \leq n$.

With respect to $\tilde \Phi _1$, we have the following result.
\begin{proposition}\label{pp0}
For $0\leq m\leq n \leq 2m$,  we have
$d_1(\alpha_{m,n} )=0.$
\end{proposition}

We prove Proposition \ref{pp0} by using the following  lemma.

\begin{lemma}\label{p2}
For $0\leq  j \leq m$, we have 
\begin{align*}
\alpha_{m,2m-j} | _{q^{1/2}=1}=4^{j} \begin{pmatrix}
      m   \\
      j
\end{pmatrix}.
\end{align*}
\end{lemma}
\begin{proof}[Proof of Proposition \ref{pp0} assuming Lemma \ref{p2}]
By Lemma \ref{p2}, for $0\leq m\leq n \leq 2m$, we have $\alpha_{m,n} \not =0$ and $\alpha_{m,n} | _{q^{1/2}=1}\not =0$,
which implies $d_1(\alpha_{m,n} )=0$.
\end{proof}

\begin{proof}[Proof of Lemma \ref{p2}]
Set $\tilde \alpha_{m,j}=\alpha_{m,2m-j}| _{q^{1/2}=1}$.
It is enough to prove 
\begin{align}
j\tilde \alpha_{m,j}-4 (m-j+1) \tilde \alpha_{m,j-1}=0,\label{en}
\end{align}
which   implies
\begin{align*}
\tilde \alpha_{m,j}&=4\frac{(m-j+1)}{j}\tilde \alpha_{m,j-1}
\\
&=\cdots =4^j\frac{(m-j+1)(m-j+2)\cdots(m)}{j(j-1)\cdots 1}\tilde \alpha_{m,0}=
4^{j} \begin{pmatrix}
      m   \\
      j
\end{pmatrix},
\end{align*}
where $\tilde \alpha_{m,0}=1$ by Proposition \ref{p1} (ii).

We prove (\ref{en}).
Note that 
\begin{align*}
\tilde \alpha_{m,j}=\alpha_{m,2m-j}| _{q^{1/2}=1}=\sum_{k=0}^{\lfloor m- \frac{2m-j}{2}\rfloor } F(m,2m-j,k)=\sum_{k=0}^{\lfloor \frac{j}{2}\rfloor } \tilde F(m,j,k),
\end{align*}
with
\begin{align*}
F(m,n,k)&=(-1)^{k}\begin{pmatrix}2m+1 \\ k\end{pmatrix}\begin{pmatrix}2m+n-2k+1 \\ 2n+1\end{pmatrix},
\\
\tilde F(m,j,k)&=F(m,2m-j,k).
\end{align*}

If we can find  $G(m,j,k)\in \mathbb{Q}$ for $0\leq k\leq \lfloor\frac{j}{2} \rfloor+1$ such that
\begin{align}
j \tilde F(m,j,k)-4 (m-j+1) \tilde F(m,j-1,k)=G(m,j,k+1)-G(m,j,k)\label{en2}
\end{align}
and  
\begin{align}
G(m,j,0)=G(m,j,\lfloor \frac{j}{2} \rfloor+1)=0,\label{en3}
\end{align}
then (\ref{en}) follows from
\begin{align*}
j\tilde \alpha_{m,j}-4 (m-j+1) \tilde \alpha_{m,j-1}=G(m,j,\lfloor \frac{j}{2} \rfloor+1)-G(m,j,0)=0.
\end{align*}
Actually we can find such $G(m,j,k)$  by using Zeilberger's algorithm \cite{P} as follows.
\begin{align*}
G(m,j,k)=\begin{cases}
-\frac{2k(2k-4m+j-3)(2k-4m+j-2)}{(4m-2j+2)(4m-2j+3)}\tilde F(m,j,k) \quad \text{for } 1\leq k\leq \lfloor\frac{j}{2} \rfloor,
\\
0 \quad  \quad  \quad \quad \quad \quad \quad \quad \quad \quad \quad \quad \quad  \quad  \quad  \quad  \quad \text{for } k=0,  \lfloor\frac{j}{2} \rfloor+1. 
\end{cases}
\end{align*}
We can check (\ref{en2}) and (\ref{en3}) by straightforward calculations. 
Hence we have the assertion.
\end{proof}

See Table \ref{table:t1} for the behavior of $d_1(\alpha_{m,n} )$, where we color the boxes gray for $0\leq n<m$ and the blanks mean $\alpha_{m,n} =0$.

\begin{table}
\begin{center}
\begin{tabular}{|c|c|c|c|c|c|c|c|c|c|c|c|}
\hline
$m \setminus n$&0&1&2 &3 &4&5&6&7&8&9&10  \\ \hline
0&0&& & &&&&&&&  \\ \hline
1&\cellcolor[gray]{0.8}&0&0 & &&&&&&&  \\ \hline
2&\cellcolor[gray]{0.8}&\cellcolor[gray]{0.8}2&0&0&0&&&&&&\\ \hline
3&\cellcolor[gray]{0.8}& \cellcolor[gray]{0.8}&\cellcolor[gray]{0.8}2 &0&0&0&0&&&&\\ \hline
4&\cellcolor[gray]{0.8}&\cellcolor[gray]{0.8}&\cellcolor[gray]{0.8}4&\cellcolor[gray]{0.8}2& 0&0&0&0&0&&\\ \hline
5&\cellcolor[gray]{0.8}&\cellcolor[gray]{0.8}&\cellcolor[gray]{0.8}&\cellcolor[gray]{0.8}4&\cellcolor[gray]{0.8}2&0&0&0&0&0&0\\ \hline
6&\cellcolor[gray]{0.8}&\cellcolor[gray]{0.8}&\cellcolor[gray]{0.8}&\cellcolor[gray]{0.8}6&\cellcolor[gray]{0.8}4&\cellcolor[gray]{0.8}2&0&0&0&0&0\\ \hline
7&\cellcolor[gray]{0.8}&\cellcolor[gray]{0.8}&\cellcolor[gray]{0.8}&\cellcolor[gray]{0.8}&\cellcolor[gray]{0.8}6&\cellcolor[gray]{0.8}4&\cellcolor[gray]{0.8}2&0&0&0&0\\ \hline
8&\cellcolor[gray]{0.8}&\cellcolor[gray]{0.8}&\cellcolor[gray]{0.8}&\cellcolor[gray]{0.8}&\cellcolor[gray]{0.8}8&\cellcolor[gray]{0.8}6&\cellcolor[gray]{0.8}4&\cellcolor[gray]{0.8} 2&0&0&0\\ \hline
\end{tabular}
\end{center}
\caption{$d_1(\alpha_{m,n} )$} \label{table:t1}
\end{table}
\begin{table}
\begin{center}
\begin{tabular}{|c|c|c|c|c|c|c|c|c|c|c|c|}
\hline
$m \setminus n$&0&1&2 &3 &4&5&6&7&8&9&10  \\ \hline
0& 0&&&&&&&&&&  \\ \hline
1&\cellcolor[gray]{0.8}&1&0&&&&&&&&  \\ \hline
2&\cellcolor[gray]{0.8}&\cellcolor[gray]{0.8}1&0 & 1&0 &&&&&& \\ \hline
3&\cellcolor[gray]{0.8}&\cellcolor[gray]{0.8}&\cellcolor[gray]{0.8}2&1&0&1&0&&&&  \\ \hline
4&\cellcolor[gray]{0.8}&\cellcolor[gray]{0.8}&\cellcolor[gray]{0.8}2&\cellcolor[gray]{0.8}1& 0&1&0&1&0&&\\ \hline
5&\cellcolor[gray]{0.8}&\cellcolor[gray]{0.8}&\cellcolor[gray]{0.8}& \cellcolor[gray]{0.8}3&\cellcolor[gray]{0.8}2&1&0&1&0&1&0\\ \hline
6&\cellcolor[gray]{0.8}&\cellcolor[gray]{0.8}&\cellcolor[gray]{0.8}& \cellcolor[gray]{0.8}3&\cellcolor[gray]{0.8}2&\cellcolor[gray]{0.8}1&0&1&0&1&0\\ \hline
7&\cellcolor[gray]{0.8}&\cellcolor[gray]{0.8}&\cellcolor[gray]{0.8}&\cellcolor[gray]{0.8}& \cellcolor[gray]{0.8}4&\cellcolor[gray]{0.8}3&\cellcolor[gray]{0.8}2&1&0&1&0\\ \hline
8&\cellcolor[gray]{0.8}&\cellcolor[gray]{0.8}&\cellcolor[gray]{0.8}&\cellcolor[gray]{0.8}& \cellcolor[gray]{0.8}4&\cellcolor[gray]{0.8}3&\cellcolor[gray]{0.8}2&\cellcolor[gray]{0.8}1&0&1&0\\ \hline
\end{tabular}
\end{center}
\caption{$d_2(\alpha_{m,n} )$} \label{t2}
\end{table}

\begin{table} \label{t3}
\begin{center}
\begin{tabular}{|c|c|c|c|c|c|c|c|c|c|c|c|c|c|c|}
\hline
$m \setminus n$& 0 &1&2 &3 &4&5&6&7&8&9&10 &11&12&13  \\ \hline
0 &0& & &&&&&&& &&& & \\ \hline
1&\cellcolor[gray]{0.8} &0&0 & &&&&&&& &&& \\ \hline
2&\cellcolor[gray]{0.8}&\cellcolor[gray]{0.8}0&0&0&0&&&&&&&&&\\ \hline
3&\cellcolor[gray]{0.8}&\cellcolor[gray]{0.8}&\cellcolor[gray]{0.8}1&0&1&1&0&&&&  &&&\\ \hline
4&\cellcolor[gray]{0.8}&\cellcolor[gray]{0.8}&\cellcolor[gray]{0.8}2&\cellcolor[gray]{0.8}2&0&0&1&0&0&&&&&\\ \hline
5&\cellcolor[gray]{0.8}&\cellcolor[gray]{0.8}&\cellcolor[gray]{0.8}&\cellcolor[gray]{0.8}2&\cellcolor[gray]{0.8}0&0&0&0&0&0&0&&&\\ \hline
6&\cellcolor[gray]{0.8}&\cellcolor[gray]{0.8}&\cellcolor[gray]{0.8}&\cellcolor[gray]{0.8}2&\cellcolor[gray]{0.8}2&\cellcolor[gray]{0.8}1&0&1&1&0& 1 &1&0&\\ \hline
7&\cellcolor[gray]{0.8}&\cellcolor[gray]{0.8}&\cellcolor[gray]{0.8}&\cellcolor[gray]{0.8}&\cellcolor[gray]{0.8}2&\cellcolor[gray]{0.8}2&\cellcolor[gray]{0.8}2&0&0&1&0&0&1&0\\ \hline
8&\cellcolor[gray]{0.8}&\cellcolor[gray]{0.8}&\cellcolor[gray]{0.8}&\cellcolor[gray]{0.8}&\cellcolor[gray]{0.8}2&\cellcolor[gray]{0.8}2&\cellcolor[gray]{0.8}2&\cellcolor[gray]{0.8}0&0& 0&0&0&0&0 \\ \hline
\end{tabular}
\end{center}
\caption{$d_3(\alpha_{m,n} )$}
\end{table}

\begin{table} \label{t4}
\begin{center}
\begin{tabular}{|c|c|c|c|c|c|c|c|c|c|c|c|c|c|c|c|c|}
\hline
$m \setminus n$&0&1&2 &3 &4&5&6&7&8&9&10 &11&12&13 &14& 15\\ \hline
0&0&&&&&&&&&&&&&&&\\ \hline
1&\cellcolor[gray]{0.8}&1&0&&&&&&&&&&&&&\\ \hline
2&\cellcolor[gray]{0.8}&\cellcolor[gray]{0.8}1&0&1&0 &&&&&& &&&&&\\ \hline
3&\cellcolor[gray]{0.8}&\cellcolor[gray]{0.8}&\cellcolor[gray]{0.8}1&1& 0&1&0&&&&  &&&&&\\ \hline
4&\cellcolor[gray]{0.8}&\cellcolor[gray]{0.8}&\cellcolor[gray]{0.8}1&\cellcolor[gray]{0.8}1& 0&2&1&1&0&&  &&&&&\\ \hline
5&\cellcolor[gray]{0.8}&\cellcolor[gray]{0.8}&\cellcolor[gray]{0.8}&\cellcolor[gray]{0.8}2&\cellcolor[gray]{0.8}2&1&0&2&1&1& 0 &&&&&\\ \hline
6&\cellcolor[gray]{0.8}&\cellcolor[gray]{0.8}&\cellcolor[gray]{0.8}&\cellcolor[gray]{0.8}2&\cellcolor[gray]{0.8}2&\cellcolor[gray]{0.8}1&0&1&0&1& 0 &1&0&&&\\ \hline
7&\cellcolor[gray]{0.8}&\cellcolor[gray]{0.8}&\cellcolor[gray]{0.8}&\cellcolor[gray]{0.8}&\cellcolor[gray]{0.8}2&\cellcolor[gray]{0.8}3&\cellcolor[gray]{0.8}1&1&0&1& 0 &1&0&1&0&\\ \hline
8&\cellcolor[gray]{0.8}&\cellcolor[gray]{0.8}&\cellcolor[gray]{0.8}&\cellcolor[gray]{0.8}&\cellcolor[gray]{0.8}2&\cellcolor[gray]{0.8}3&\cellcolor[gray]{0.8}1&\cellcolor[gray]{0.8}1&0&2&1&1&0&2 &1 &1\\ \hline
\end{tabular}
\end{center}
\caption{$d_4(\alpha_{m,n} )$}
\end{table}
\begin{table}
\begin{center}
\begin{tabular}{|c|c|c|c|c|c|c|c|c|c|c|c|c|c|c|c|c|}
\hline
$m \setminus n$&0&1&2 &3 &4&5&6&7&8&9&10 &11&12&13 &14& 15\\ \hline
0&0& & &&&&&&& &&& && &\\ \hline
1&\cellcolor[gray]{0.8}&0&0 & &&&&&&& &&& &&\\ \hline
2&\cellcolor[gray]{0.8}&\cellcolor[gray]{0.8}1&0&0&0 &&&&&& &&&&&\\ \hline
3&\cellcolor[gray]{0.8}&\cellcolor[gray]{0.8}&\cellcolor[gray]{0.8}0&0& 0&0&0&&&&  &&&&&\\ \hline
4&\cellcolor[gray]{0.8}&\cellcolor[gray]{0.8}&\cellcolor[gray]{0.8}0&\cellcolor[gray]{0.8}0&0&0&0&0&0&&  &&&&&\\ \hline
5&\cellcolor[gray]{0.8}&\cellcolor[gray]{0.8}&\cellcolor[gray]{0.8}&\cellcolor[gray]{0.8}1&\cellcolor[gray]{0.8}1&0&1&1&1&1& 0 &&&&&\\ \hline
6&\cellcolor[gray]{0.8}&\cellcolor[gray]{0.8}&\cellcolor[gray]{0.8}&\cellcolor[gray]{0.8}1&\cellcolor[gray]{0.8}1&\cellcolor[gray]{0.8}1&0&0&1&1& 1 &0&0&&&\\ \hline
7&\cellcolor[gray]{0.8}&\cellcolor[gray]{0.8}&\cellcolor[gray]{0.8}&\cellcolor[gray]{0.8}&\cellcolor[gray]{0.8}2&\cellcolor[gray]{0.8}2&\cellcolor[gray]{0.8}1&0&0& 0 &1&1&0&0&0 &\\ \hline
8&\cellcolor[gray]{0.8}&\cellcolor[gray]{0.8}&\cellcolor[gray]{0.8}&\cellcolor[gray]{0.8}&\cellcolor[gray]{0.8}2&\cellcolor[gray]{0.8}2&\cellcolor[gray]{0.8}2&\cellcolor[gray]{0.8}0&0&0&0&0&0&0 &0 &0\\ \hline
9&\cellcolor[gray]{0.8}&\cellcolor[gray]{0.8}&\cellcolor[gray]{0.8}&\cellcolor[gray]{0.8}&\cellcolor[gray]{0.8}&\cellcolor[gray]{0.8}2&\cellcolor[gray]{0.8}2&\cellcolor[gray]{0.8}0&\cellcolor[gray]{0.8}0&0& 0 &0&0&0&0&0\\ \hline
10&\cellcolor[gray]{0.8}&\cellcolor[gray]{0.8}&\cellcolor[gray]{0.8}&\cellcolor[gray]{0.8}&\cellcolor[gray]{0.8}&\cellcolor[gray]{0.8}2&\cellcolor[gray]{0.8}3&\cellcolor[gray]{0.8}2&\cellcolor[gray]{0.8}1&\cellcolor[gray]{0.8}1&0&1&1&1 &1 &0\\ \hline

\end{tabular}
\end{center}
\caption{$d_5(\alpha_{m,n} )$} \label{t5}
\end{table}

For $l\geq 2$, we can compute $d_l(\alpha_{m,n} )$ for small $m,n\geq 0$ though we do not have a general result.
See Table \ref{t2}--\ref{t5} for $d_l(\alpha_{m,n} )$ for $l=2,\ldots, 5$.

We have the following conjecture, which enables us to compute $d_{l}(\alpha_{m,n})$ for general $m,n\geq 0$ by using explicit computations of $d_{l}(\alpha_{m,n})$
for small $m,n\geq 0$. 
\begin{conjecture}
For $l\geq 1$,
in the range  $0\leq m\leq n\leq 2m$, 
$d_{l}(\alpha_{m,n})$ is periodic with period $l$ both in $m$ and in $n$, i.e., 
for $m\equiv \tilde m$  $(\mod l)$, and  
 $n\equiv \tilde n$ $(\mod l)$, we have
\begin{align*}
d_{l}(\alpha_{m,n})=d_{l}(\alpha_{\tilde m, \tilde n}).
\end{align*}
\end{conjecture}
We also have the following conjecture. 
\begin{conjecture}
For a prime $l\geq 1$ and $0\leq m\leq n\leq 2m$, 
we have $d_{l}(\alpha_{m,n})\in \{0, 1\}$. 
\end{conjecture}

\section{Colored Jones polynomial of  Milnor's link}\label{Ex2}
In this section, we give examples of  computations of the reduced colored Jones polynomials of Milnor's link.

For $n\geq 3,$ let   $A_n$  be  the $n$-component  Milnor's link  depicted in Figure \ref{fig:brunnianex}.
Let $A_2=H$ be the Hopf link  depicted in Figure \ref{fig:borro}.
Note that $A_3=B$ is the Borromean rings, and $A_n=B(A_{n-1}; 1)$ for $n \geq 3$.
\begin{figure}
\centering
\begin{picture}(50,60)
\put(10,0){\includegraphics[width=2cm,clip]{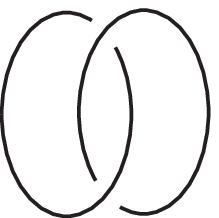}}
\put(-20,25){$H=$}
\end{picture}
\caption{ The Hopf link $H$  }\label{fig:borro}
\end{figure}%
\begin{figure}
\centering
\begin{picture}(140,60)
\put(10,0){\includegraphics[width=5cm,clip]{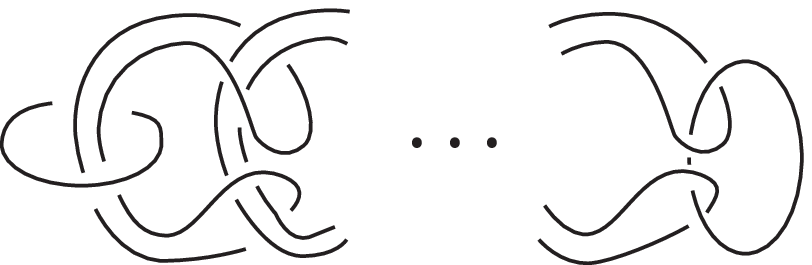}}
\put(-23,20){$A_n=$}
\end{picture}
\caption{Milnor's link $A_n$}\label{fig:brunnianex}
\end{figure}%

We use the following result.
\begin{lemma}[Habiro {\cite[Corollary 14.2]{H2}}]\label{le1}
For $i,j,k \geq 0$, we have
\begin{align*}
J_{B; P'_i,P'_j, P'_k}=\begin{cases} (-1)^i\{2i+1\}_{i+1}/\{1\} \quad  \ \text{if} \quad  i=j=k,
\\
0 \quad  \quad \quad \quad \quad \quad \quad  \quad \quad \quad \text{otherwise}. 
\end{cases}
\end{align*}
\end{lemma}

By Theorem \ref{1} and Lemma \ref{le1}, we can compute the reduced colored Jones polynomial of $A_n$ as follows.
\begin{proposition}\label{l2}
For $n\geq 3$, we have
\begin{align*}
J_{A_n;  P_{l_1}',\ldots, P_{l_n}'}&=\delta_{l_1,l_2} \delta_{l_{n-1},l_{n}}\s{l_2}{l_2}{l_3}\s{l_3}{l_3}{l_4}\cdots \s{l_{n-2}}{l_{n-2}}{l_{n-1}}J_{B;  P_{l_{n-1}}',P_{l_{n-1}}',P_{l_{n-1}}'}
\\
&=\delta_{l_1,l_2} \delta_{l_{n-1},l_{n}}\s{l_2}{l_2}{l_3}\s{l_3}{l_3}{l_4}\cdots \s{l_{n-2}}{l_{n-2}}{l_{n-1}} (-1)^{l_{n-1}}\{2l_{n-1}+1\}_{l_{n-1}+1}/\{1\}. 
\end{align*}
\end{proposition}
\begin{proof}
We use induction on $n$.
For $n=3$, by Lemma \ref{le1}, we have
\begin{align*}
J_{B;  P_{l_1}', P_{l_2}', P_{l_3}'}&=\delta_{l_1,l_2} \delta_{l_2,l_3} J_{B;  P_{l_{2}}',P_{l_{2}}',P_{l_{2}}'}.
\end{align*}
For $n>3$, by the assumption of induction, we have
\begin{align*}
J_{A_n;  P_{l_1}',\ldots, P_{l_n}'}&=\sum_{k_1\geq 0}\s{l_1}{l_2}{k_1}J_{A_{n-1};  P_{k_1}',P_{l_3}',\ldots, P_{l_n}'}
\\
&=\delta_{l_1,l_2}\sum_{k_1\geq 0}\s{l_2}{l_2}{k_1}J_{A_{n-1};  P_{k_1}',P_{l_3}',\ldots, P_{l_n}'}
\\
&=\delta_{l_1,l_2}\sum_{k_1\geq 0}\s{l_2}{l_2}{k_1}\delta_{k_1,l_3} \delta_{l_{n-1},l_{n}}\s{l_3}{l_3}{l_4}\s{l_4}{l_4}{l_5}\cdots \s{l_{n-2}}{l_{n-2}}{l_{n-1}}J_{A;  P_{l_{n-1}}',P_{l_{n-1}}',P_{l_{n-1}}'}
\\
&=\delta_{l_1,l_2}  \delta_{l_{n-1},l_{n}}\s{l_1}{l_2}{l_3}\s{l_3}{l_3}{l_4}\s{l_4}{l_4}{l_5}\cdots \s{l_{n-2}}{l_{n-2}}{l_{n-1}}J_{A;  P_{l_{n-1}}',P_{l_{n-1}}',P_{l_{n-1}}'}.
\end{align*}
Thus we have the assertion.
 \end{proof}
 We have the following corollaries.
 \begin{corollary}\label{cm}
For $n\geq 3$, we have
\begin{align*}
J_{A_n;  P_1',\ldots, P_1'}&=(-1)^{n}\tilde \Phi _1^{n-2} \tilde \Phi _2^{n-2} \tilde \Phi _3 \tilde \Phi_4^{n-3}.
\end{align*}
\end{corollary}
\begin{proof}
The assertion follows from Proposition \ref{l2} and
\begin{align}
\s{1}{1}{1}= -\{4\}=-\tilde \Phi_4 \tilde \Phi_2 \tilde \Phi_ 1
\end{align} 
and 
\begin{align}
J_{B; P_{1}',P_{1}',P_{1}'}=-\{3\}_{2}/\{1\}=-\tilde \Phi_3\tilde \Phi_2\tilde \Phi_1.
\end{align}

\end{proof}

\begin{corollary}
\begin{itemize}
\item[\rm{(i)}]
For $n\geq 3$ and $l_1,\ldots,l_n\geq0,$
 unless   $l_1= l_2$, $l_{n-1}=l_{n}$ and 
  unless  $\frac{1}{2} \leq \frac{l_{i+1}} {l_i}\leq 2$ for $2\leq i \leq n-2,$
  we have $J_{A_n;  P_{l_1}',\ldots, P_{l_n}'}=0$.
\item[\rm{(ii)}]
For $a_0\geq 0$ and $a_i=2a_{i-1}$ for $i=1,\ldots, n-3$, we have
\begin{align*}
J_{A_{n};  P_{a_0}',P_{a_0}',P_{a_{1}}',\ldots, P_{a_{n-4}}', P_{a_{n-3}}', P_{a_{n-3}}'}=\big(\prod _{m=1}^{n-3} (-1)^{a_m}\{2a_m\}! \big) \big((-1)^{a_{n-3}}\{2a_{n-3}+1\}_{a_{n-3}+1}/\{1\} \big).
\end{align*}
\item[\rm{(iii)}]
For $b_0\geq 0$ and  $b_i=2b_{i-1}-1$  for $i=1,\ldots, n-3$  $($i.e., $b_i=2^ib_0-\frac{i(i+1)(2i+1)}{6})$ we have
\begin{align*}
J_{A_{n};  P_{b_0}',P_{b_0}',P_{b_{1}}',\ldots,  P_{b_{n-4}}', P_{b_{n-3}}', P_{b_{n-3}}'}&
\\
=\big(\prod _{m=1}^{n-3}(-1)^{b_m}\{2b_m-1\}!&\{4b_m\}/\{1\} \big) \big((-1)^{b_{n-3}}\{2b_{n-3}+1\}_{b_{n-3}+1}/\{1\} \big).
\end{align*}

\end{itemize}
\end{corollary}
\begin{proof}
The assertions (i), (ii) and  (iii)  follow from  Corollary \ref{c1} (i), (ii) and (iii), respectively, and Proposition \ref{l2}.

\end{proof}
\section{Proofs}\label{pr}
In this section,  we prove Theorems \ref{1}, \ref{3}, and Proposition \ref{p3}.

\subsection{Proof of Theorem \ref{1}} \label{proof}
We prove Theorem \ref{1}. 
We also prove Lemma \ref{l0} at the end of this section.

In \cite{H2}, Habiro defined the element 
\begin{align*}
S_n=\prod_{i=1}^{n}(V_2-(q^i+1+q^{-i}))\in \mathcal{R},
\end{align*} 
for $n\geq 0$,  which is a kind of dual of $P'_n$ with respect to the symmetric bilinear form 
$J_{H; -, -}\co \mathcal{R} \times \mathcal{R} \rightarrow \mathbb{Q}(q^{1/2})$
as follows.
\begin{lemma}[{Habiro \cite[Proposition 6.6]{H2}}]\label{le2}
For $m,n\geq 0,$ we have
\begin{align*}
 J_{H;  P'_m, S_n}=\delta _{m,n}\{2m+1\}_{2m}/\{m\}!.
\end{align*}
\end{lemma}

 Recall the element $\alpha_{m,n}=\alpha_{m,n}(q^{1/2})\in \mathbb{Z}[q^{1/2},q^{-1/2}]$ defined in Theorem \ref{1} (ii).
 We reduce Theorem \ref{1} to the following  proposition.
\begin{proposition}\label{pro}
For $l\geq 0$, we have
\begin{align*}
S_l=\sum_{m\geq 0}\alpha_{l,m}\{m\}!P'_m.
\end{align*}
\end{proposition}
\begin{proof}[Proof of Theorem \ref{1} assuming Proposition \ref{pro}]
By   Proposition \ref{pro} and  Lemma \ref{le1}, we have 
\begin{align*}
J_{B;P'_i,P'_j, S_l}&=\sum_{m\geq 0} \alpha_{l,m}\{m\}!J_{B; P'_i,P'_j, P'_m} 
\\
&=\delta_{i,j}\alpha_{l,i}\{i\}!(-1)^i\{2i+1\}_{i+1}/\{1\} \\
&=\delta_{i,j}\alpha_{l,i}(-1)^i\{2i+1\}!/\{1\}.
\end{align*}
On the other hand, since $B=B(H; 1)$, we have
\begin{align*}
J_{A;P'_i,P'_j, S_l}&=J_{B(H;1);P'_i,P'_j, S_l}
\\
&=\sum_{k\geq 0}x_{{i,}{j}}^{({k})}J_{H;P'_k, S_l}
\\
&=x_{{i,}{j}}^{({l})}\{2l+1\}_{2l}/\{l\}!.
\end{align*}
Here, the last identity follows from Lemma  \ref{le2}.

Consequently, we have
\begin{align*}
x_{{i,}{j}}^{({l})}=\delta_{i,j}(-1)^i\frac{\{2i+1\}!\{l\}!}{\{2l+1\}!}\alpha_{l,i},
\end{align*}
which completes the proof.
\end{proof}
In what follows, we prove Proposition \ref{pro}.
We use two more lemmas  as follows.
\begin{lemma}\label{le3}
For $l\geq 0,$ we have
\begin{align*}
S_l=\sum_{k=0}^l(-1)^k\begin{bmatrix}2l+1 \\ k \end{bmatrix}V_{2l-2k}.
\end{align*}
\end{lemma}
\begin{proof}
We use an induction on $l$.
For $m=0,1,$ we have
\begin{align*}
S_0&=1,
\\
S_1&=(V_2-(q+1+q^{-1}))=V_2-[3].
\end{align*}
For $m\geq 2$, we have
\begin{align*}
S_{m}=&S_{m-1}(V_2-(q^{m}+1+q^{-m}))
\\
=&\sum_{k=0}^{m-1}(-1)^k\begin{bmatrix}2m-1 \\ k \end{bmatrix}V_{2m-2k-2}(V_2-(q^{m}+1+q^{-m}))
\\
=&\sum_{k=0}^{m-2}(-1)^k\begin{bmatrix}2m-1 \\ k \end{bmatrix}(V_{2m-2k}+V_{2m-2k-2}+V_{2m-2k-4}-(q^{m}+1+q^{-m})V_{2m-2k-2})
\\
&+(-1)^{m-1}\begin{bmatrix}2m-1 \\ m-1\end{bmatrix}(V_2-(q^{m}+1+q^{-m}))
\\
=&\sum_{k=0}^{m-2}(-1)^k\begin{bmatrix}2m-1 \\ k \end{bmatrix}(V_{2m-2k}-(q^{m}+q^{-m})V_{2m-2k-2}+V_{2m-2k-4})
\\&+(-1)^{m-1}\begin{bmatrix}2m-1 \\ m-1\end{bmatrix}(V_2-(q^{m}+1+q^{-m}))
\\
=&V_{2m}-(q^m+q^{-m})V_{2m-2}-[2m-1]V_{2m-2}
\\&+\sum_{k=2}^{m-1}\Big((-1)^{k-2}\begin{bmatrix}2m-1 \\ k-2 \end{bmatrix}-(-1)^{k-1}\begin{bmatrix}2m-1 \\ k-1 \end{bmatrix}(q^{m}+q^{-m})
+(-1)^{k}\begin{bmatrix}2m-1 \\ k \end{bmatrix}\Big)V_{2m-2k}
\\&+(-1)^{m-2}\begin{bmatrix}2m-1 \\ m-2 \end{bmatrix}-(-1)^{m-1}\begin{bmatrix}2m-1 \\ m-1\end{bmatrix}(q^{m}+1+q^{-m})
\\
=&\sum_{k=1}^{m}(-1)^k\begin{bmatrix}2m+1 \\ k \end{bmatrix}V_{2m-2k}.
\end{align*}
\end{proof}
The following lemma is observed  in the proof of \cite[Proposition 6.6]{H2}.
\begin{lemma}[Habiro \cite{H2}]\label{le4}
For $m,n\geq 0,$ we have
\begin{align*}
J_{H; V_m, S_n}= \{m+n+1\}_{2n+1}/\{1\}.
\end{align*}
\end{lemma}
Now, we prove Proposition \ref{pro}.
\begin{proof}[Proof of Proposition \ref{pro}]
For $m,n\geq 0,$ we have
\begin{align}
\begin{split}\label{luck}
J_{H; S_m, S_n}&=\sum_{l\geq 0}  \alpha_{m,l}\{l\}! J_{H; P'_l, S_n}
\\
&=\alpha_{m,n}\{n\}! (\{2n+1\}_{2n}/ \{n\}!)
\\
&=\alpha_{m,n}\{2n+1\}_{2n}.
\end{split}
\end{align}
Hence we have
\begin{align*}
\alpha_{m,n}&=J_{H; S_m,S_n} / \{2n+1\}_{2n}
\\
&=\sum_{k=1}^m(-1)^k\begin{bmatrix}2m+1 \\ k \end{bmatrix}J_{H;  V_{2m-2k},S_n }/\{2n+1\}_{2n}
\\
&=\sum_{k=1}^m(-1)^k\begin{bmatrix}2m+1 \\ k \end{bmatrix}\{2m+n-2k+1\}_{2n+1}/\{2n+1\}!
\\
&=\sum_{k=1}^{m}(-1)^k\begin{bmatrix}2m+1 \\ k \end{bmatrix}\begin{bmatrix}2m+n-2k+1\\ 2n+1\end{bmatrix},
\end{align*}
where the second identity follows from Lemma \ref{le3} and the third identity follows from Lemma \ref{le4}.
Hence we have the assertion.
\end{proof}

We prove Lemma \ref{l0}.
\begin{proof}[Proof of Lemma \ref{l0}]
By (\ref{luck}) and symmetric property of $J_{H; S_m, S_n}$,   we have
\begin{align*}
J_{H; S_m, S_n}&=\alpha_{m,n}\{2n+1\}_{2n}
\\
&=\alpha_{n,m}\{2m+1\}_{2m}
\end{align*}
for $m,n\geq 0$.
Thus we have
\begin{align*}
\alpha_{m,n}
&=\frac{\{2m+1\}_{2m}}{\{2n+1\}_{2n}}\alpha_{n,m}
\\
&=\frac{\{2m+1\}!}{\{2n+1\}!}\alpha_{n,m}.
\end{align*}
Hence we have the assertion.
\end{proof}
\subsection{Proof of Theorem \ref{3.1}}\label{pr4}
We reduce Theorem \ref{3.1} to Proposition \ref{3} as follows.
\begin{proposition}\label{3}
Let $L=L_1 \cup \cdots \cup L_n$ be an algebraically-split, $0$-framed link in $S^3$ and  set $\check L=L_2 \cup L_3\cup \cdots \cup L_n$.
For $\epsilon, \epsilon', \epsilon_2,\ldots,\epsilon_{n}\in \{\pm 1\}$, we have
\begin{align*}
J_{S^3(B(L;1); \epsilon, \epsilon',  \epsilon_2, \ldots,\epsilon_{n})}-J_{S^3(\check L; \epsilon_2,\ldots,\epsilon_{n})} \in \Phi_1^2\Phi_2^2\Phi_3\Phi_4\Phi_6\widehat{\mathbb{Z}[q]}.
\end{align*}
\end{proposition}
\begin{proof}[Proof of Theorem \ref{3.1}  assuming Proposition \ref{3}]
Let $M$ and $M'$ be integral homology spheres related by a special Bing double surgery, i.e.,
$M'$ is orientation-preserving homeomorphic to $M(B(K); \epsilon, \epsilon')$  for a $0$-framed  knot $K$ and $\epsilon, \epsilon'\in \{\pm 1\}$.
Let  $T=T_1 \cup \cdots \cup T_n$ be an algebraically split  link  in $S^3$ such that $S^3(T; \epsilon_1,\ldots,\epsilon_{n})$, $\epsilon_1,\ldots,\epsilon_{n}\in \{\pm 1\}$, is orientation-preserving homeomorphic to $M$.
Here, by isotopy of $K$ in $M$, we can assume  that $K$ is null-homologous in $M\setminus N$, where $N$ is the union of solid tori 
on which the surgery operation along $T$ was done.
Now, $K$ may be regarded as a $0$-framed knot   in $S^3\setminus T$.
Set  $L=K\cup T$. 
By Proposition \ref{3}, we have 
\begin{align*}
J_{M(B(K); \epsilon, \epsilon')}-J_{M}&=J_{S^3(B(L;1); \epsilon, \epsilon',  \epsilon_1, \ldots,\epsilon_{n})}-J_{S^3(\check L=T; \epsilon_1,\ldots,\epsilon_{n})}
\\
& \in \Phi_1^2\Phi_2^2\Phi_3\Phi_4\Phi_6\widehat{\mathbb{Z}[q]}.
\end{align*}
Hence we have the assertion.
\end{proof}
We prove Proposition \ref{3}.
We use the following lemma.
\begin{lemma}[Habiro {\cite[Theorem 8.2]{H2}}]\label{h}
Let $L$ be an $n$-component, algebraically-split link with $0$-framing. For $l_1,\ldots , l_n\geq 0$, we have
\begin{align*}
J_{L; P'_{l_1},\ldots , P'_{l_n}}\in \frac{\{ 2l_j+1\}_{l_j+1}}{\{1\}}\mathbb{Z}[q^{1/2},q^{-1/2}],
\end{align*} 
where $j$  is an integer such that   $l_j=\max\{l_i\}_{1\leq i\leq n}$.
\end{lemma}
\begin{proof}[Proof of Proposition \ref{3}]
Let $L$ be an $n$-component, algebraically-split link with $0$-framing. 
By the definition,  we have
\begin{align*}
J_{S^3(B(L; 1); \epsilon, \epsilon', \epsilon_2, \ldots, \epsilon_n)}&=J_{B(L;1); \omega^{-\epsilon}, \omega^{-\epsilon'}, \omega^{-\epsilon_2}, \ldots, \omega^{-\epsilon_n}}
\\
&=\sum_{i\geq 0}\sum_{j\geq 0}(-\epsilon)^i(-\epsilon')^jq^{-\epsilon i(i+3)/4-\epsilon' j(j+3)/4}J_{B(L;1); P_i', P_j', \omega^{-\epsilon_2}, \ldots, \omega^{-\epsilon_n}}.
\end{align*}
By Lemma \ref{10} and Theorem \ref{1}, we have
\begin{align*}
J_{B(L;1); P_i', P_j', \omega^{-\epsilon_2}, \ldots, \omega^{-\epsilon_n}}
=\delta_{i,j}\sum _{l\geq 0}\s{i}{i}{l}J_{L; P_l', \omega^{-\epsilon_2}, \ldots, \omega^{-\epsilon_n}},
\end{align*}
thus we have 
\begin{align*}
\sum_{i\geq 0}\sum_{j\geq 0}&(-\epsilon)^i(-\epsilon')^jq^{-\epsilon i(i+3)/4-\epsilon' j(j+3)/4}J_{B(L;1); P_i', P_j', \omega^{-\epsilon_2}, \ldots, \omega^{-\epsilon_n}}
%\\
%&=\sum_{i\geq 0}\sum_{j\geq 0}(-\epsilon)^i(-\epsilon')^jq^{-\epsilon i(i+3)/4-\epsilon' j(j+3)/4}\sum _{l\geq 0}\s{i}{j}{l}J_{L; P_l', \omega^{-\epsilon_2}, \ldots, \omega^{-\epsilon_n}}
\\
&=\sum_{i\geq 0}(\epsilon\epsilon')^iq^{-(\epsilon+\epsilon') i(i+3)/4}\sum _{l\geq 0}\s{i}{i}{l}J_{K; P_l',  \omega^{-\epsilon_2}, \ldots, \omega^{-\epsilon_n}}
\\
&=\sum _{l\geq 0}\big(\sum_{i\geq 0}(\epsilon\epsilon')^iq^{-(\epsilon+\epsilon') i(i+3)/4}\s{i}{i}{l}\big)J_{L; P_l', \omega^{-\epsilon_2}, \ldots, \omega^{-\epsilon_n}}
\\
&=1+\sum _{l\geq 1}
\big(\sum_{i\geq 0}(\epsilon\epsilon')^iq^{-(\epsilon+\epsilon') i(i+3)/4}\s{i}{i}{l}\big)J_{L; P_l', \omega^{-\epsilon_2}, \ldots, \omega^{-\epsilon_n}}.
\end{align*}
For $l\geq 1$,  set 
\begin{align*}
s_{l}^{(\epsilon,\epsilon')}&=\sum_{i\geq 0}(\epsilon\epsilon')^iq^{-(\epsilon+\epsilon') i(i+3)/4}\s{i}{i}{l}
\\
&=\sum_{i\geq \lceil l/2 \rceil }^{2l}(\epsilon\epsilon')^iq^{-(\epsilon+\epsilon') i(i+3)/4}\s{i}{i}{l},
\end{align*}
where the second identity follows from Corollary \ref{c1} (i).
It is enough to prove
\begin{align*}
s_{l}^{(\epsilon,\epsilon')}J_{L; P_l', P'_{l_2},\ldots, P'_{l_n}}\in 
\tPhi_1^2\tPhi_2^2\tPhi_3\tPhi_4\tPhi_6\mathbb{Z}[q^{1/2},q^{-1/2}],
\end{align*}
for $l\geq 1$ and $l_2,\ldots, l_n \geq 0$.

By  Lemma \ref{h}, we have
\begin{align*}
&J_{L; P_1', P'_{l_2},\ldots, P'_{l_n}}\in  \frac{\{ 3\}_{2}}{\{1\}}\mathbb{Z}[q^{1/2},q^{-1/2}]\subset \tPhi_1\tPhi_2\tPhi_3\mathbb{Z}[q^{1/2},q^{-1/2}],
\\
&J_{L; P_2', P'_{l_2},\ldots, P'_{l_n}}\in \frac{\{ 5\}_{3}}{\{1\}}\mathbb{Z}[q^{1/2},q^{-1/2}]\subset \tPhi_1^2\tPhi_2\tPhi_3\tPhi_4\tPhi_5\mathbb{Z}[q^{1/2},q^{-1/2}],
\end{align*}
and for $l\geq 3$, we have
\begin{align*}
J_{L; P_l', P'_{l_2},\ldots, P'_{l_n}}&\in \frac{\{ 2l+1\}_{ l+1}}{\{1\} }\mathbb{Z}[q,q^{-1}]
\\
& \subset \tPhi_1^2\tPhi_2^2\tPhi_3\tPhi_4\tPhi_6\mathbb{Z}[q^{1/2},q^{-1/2}].
\end{align*}
Thus we have only to prove 
\begin{align}
&s_{1}^{(\epsilon,\epsilon')}\in  \tPhi_1\tPhi_2\tPhi_4\tPhi_6\mathbb{Z}[q^{1/2},q^{-1/2}],\label{n1}
\\
&s_{2}^{(\epsilon,\epsilon')}\in  \tPhi_2\tPhi_6
\mathbb{Z}[q^{1/2},q^{-1/2}].\label{n2}
\end{align}

By Corollaries  \ref{c0} and \ref{c1}, we have 
\begin{align*}
\s{1}{1}{1}&=-\{4\}=-\tilde \Phi_1\tilde \Phi_2\tilde \Phi_4,
\\
\s{2}{2}{1}&=(-1)^3\frac{\{5\}!\{1\}!}{\{3\}!\{2\}!}\s{1}{1}{2}=\big((-1)^3\frac{\{5\}!\{1\}!}{\{3\}!\{2\}!}\big)\big(-\{2\}!\big)= \tilde\Phi_1^3\tilde\Phi_2\tilde\Phi_4\tilde\Phi_5,
\\
\s{1}{1}{2}&=-\{2\}!=-\tilde\Phi_1^2\tilde\Phi_2,
\\
\s{2}{2}{2}&=\tPhi_1^2 \tPhi_2 (q^{-5}+q^{-4}+2q^{-3}+q^{-2}+2q^{-1}+2+2q+q^2+2q^3+q^4+q^5),
\\
\s{2}{2}{3}&=(-1)^2\{3\}!\{8\}/\{1\}=\tilde\Phi_1^3\tilde\Phi_2^2\tilde\Phi_3\tilde\Phi_4\tilde\Phi_8,
\\
\s{2}{2}{4}&=\{4\}!=\tilde\Phi_1^4\tilde\Phi_2^2\tilde\Phi_3\tilde\Phi_4.
\end{align*}
Thus we have
\begin{align*}
s_{1}^{(-1,-1)}&=\sum_{i=1 }^{2}q^{i(i+3)/2}\s{i}{i}{1}
\\
&=\Phi_1\Phi_2\Phi_4\Phi_6 \cdot q(1-q+q^3),
\\
s_{2}^{(-1,-1)}&=\sum_{i=1 }^{4}q^{i(i+3)/2}\s{i}{i}{2}
\\
&=\Phi_1^2\Phi_2\Phi_3\Phi_6  \cdot q^{5/2}(1 + q + q^2 + q^3 + q^4 + q^5 + q^6 - q^{13} - q^{14} -
    q^{16} + q^{21}),
\\
s_{1}^{(1,1)}&=\sum_{i=1 }^{2}q^{-i(i+3)/2}\s{i}{i}{1}
\\
&=-\Phi_1\Phi_2\Phi_4\Phi_6 \cdot q^{-10} (-1 + q^2 +  q^3),
\\
s_{2}^{(1,1)}&=\sum_{i=1 }^{4}q^{-i(i+3)/2}\s{i}{i}{2}
\\
&=\Phi_1^2\Phi_2\Phi_3\Phi_6 \cdot q^{-61/2}(1 - q^2 - q^7 - q^8 + q^{15} + q^{16} + q^{17} + 
   q^{18} + q^{19} + q^{20} + q^{21})    ,
\\
s_{1}^{(-1,1)}=s_{1}^{(1,-1)}&=\sum_{i=1 }^{2}(-1)^i\s{i}{i}{1}
\\
&= \Phi_1\Phi_2\Phi_4\Phi_6\Phi_{12}\cdot q^{-5},
\\
s_{2}^{(-1,1)}=s_{2}^{(1,-1)}&=\sum_{n=1 }^{4}(-1)^n\s{n}{n}{2}
\\
&=\Phi_1^2\Phi_2\Phi_6 \cdot q^{-33/2} (1 + q - 
   q^3 - q^4 + q^5 + 2 q^6 + q^7 - q^8 - 2 q^9 + q^{10} + 4 q^{11} + 
   4 q^{12} 
   \\&- 3 q^{14} + 4 q^{16} + 4 q^{17} + q^{18} - 2 q^{19} - q^{20} + q^{21} + 
   2 q^{22} + q^{23} - q^{24} - q^{25} + q^{27} + q^{28}). 
\end{align*}
Hence we have (\ref{n1}) and (\ref{n2}).
This completes the proof.
\end{proof}

\subsection{Proof of Proposition \ref{p3}} \label{pr3}
To prove Proposition \ref{p3}, we  use the following lemma.
\begin{lemma}[Habiro{\cite[Proposition 14.5]{H2}}]\label{lh}
For  $i,j,k \in \mathbb{Z}$, we have
\begin{align*}
J_{M_{i,j,k}}=\sum_{l\geq 0} \omega_{i,l}\omega_{j,l}\omega_{k,l}(-1)^l \{2l+1\}_{l+1} / \{1\},
\end{align*}
where for $p\in \mathbb{Z}$ and $n\geq 0$,
\begin{align*}
\omega_{p,n}=\begin{cases}
q^{\frac{1}{4}n(n+3)}\sum _{\mathbf{i}\in S(n,p)} \begin{bmatrix}n \\ \mathbf{i}\end{bmatrix}_q q^{f(\mathbf{i})} \quad \text{for } p\geq 0,
\\
(-1)^n q^{-\frac{1}{4}n(n+3)}\sum _{\mathbf{i}\in S(n,-p)} \begin{bmatrix}n \\ \mathbf{i}\end{bmatrix}_{q-1} q^{-f(\mathbf{i})}
\quad \text{for } p\leq 0.
\end{cases}
\end{align*}
Here for $p,n \geq 0$,  we set
\begin{align*}
S(n,p)=\{(i_1,\ldots, i_p) \ | \ i_1,\ldots, i_p\geq 0, i_1+\cdots +i_p=n \},
\end{align*}
and for $\mathbf{i}=(i_1,\ldots, i_p)\in S(n,p)$,  we set
\begin{align*}
\begin{bmatrix}n \\ \mathbf{i}\end{bmatrix}_q=\frac{[n]_q!}{[i_1]_q!\cdots [i_p]_q!}, \quad 
f(\mathbf{i})=\sum_{j=1}^{p-1}(s_j^2+s_j),
 \end{align*}
where $s_j=\sum_{k=1}^j i_k$ and 
\begin{align*}
[m]_q=\frac{q^m-1}{q-1}, \quad [m]_q!=[m]_q[m-1]_q\cdots [1]_q, 
\end{align*}
for $m\geq 0$.
\end{lemma}
\begin{proof}[Proof of Proposition \ref{p3}]
By Lemma \ref{lh} and $\omega_{p,0}=1$ for $p\in \mathbb{Z}$, we have
\begin{align*}
J_{M_{i,j,k}}-1&=\sum_{l\geq 1} \omega_{i,l}\omega_{j,l}\omega_{k,l}(-1)^l \{2l+1\}_{l+1} / \{1\}
\\
&\equiv -\omega_{i,1}\omega_{j,1}\omega_{k,1}\tilde \Phi_3\tilde \Phi_2 \tilde \Phi_1+ \omega_{i,2}\omega_{j,2}\omega_{k,2}\tilde \Phi_5\tilde \Phi_4\tilde \Phi_3\tilde \Phi_2 \tilde \Phi_1^2 \quad   (\mod \tilde \Phi_2^2)
\\
&\equiv -q^{-2}\omega_{i,1}\omega_{j,1}\omega_{k,1}\Phi_3 \Phi_2  \Phi_1+ q^{-5-\frac{1}{2}}\omega_{i,2}\omega_{j,2}\omega_{k,2}\Phi_5\Phi_4 \Phi_3\Phi_2  \Phi_1^2 
\quad (\mod  \Phi_2^2).
\end{align*}
Thus we have 
\begin{align*}
\frac{J_{M_{i,j,k}}-1}{\Phi_2}&=-q^{-2}\omega_{i,1}\omega_{j,1}\omega_{k,1}\Phi_3 \Phi_1+q^{-5-\frac{1}{2}} \omega_{i,2}\omega_{j,2}\omega_{k,2}\Phi_5\Phi_4 \Phi_3 \Phi_1^2 \quad (\mod  \Phi_2).
\end{align*}
For $p\geq 0$, we have
\begin{align*}
\omega_{p,1}&=q^{2p+1}\sum_{t=1}^{p}q^{-2t},
\\
\omega_{-p,1}&=-q^{-2p-1}\sum_{t=1}^{p}q^{2t},
\\
\omega_{p,2}&=q^{\frac{5}{2}} \Big(\sum_{t=1}^p q^{6(p-t)} +[2]_q\sum_{1\leq s < t\leq p} q^{2(t-s)+6(p-t)}\Big),
\\
\omega_{-p,2}&=q^{-\frac{5}{2}} \Big(\sum_{t=1}^p q^{-6(p-t)} +[2]_{q^{-1}}\sum_{1\leq s < t\leq p} q^{-2(t-s)-6(p-t)}\Big).
\end{align*}
Thus, for $p\in \mathbb{Z}$, we have
\begin{align*}
\omega_{p,1} |_{q=-1}&=-p,
\quad q^{-\frac{1}{2}}\omega_{p,2}|_{q=-1}=p.
\end{align*}
Together with
\begin{align*}
\Phi_1|_{q=-1}=-2,\quad  \Phi_3|_{q=-1}=1, \quad  \Phi_4|_{q=-1}=2, \quad \Phi_5|_{q=-1}=1,
\end{align*}
we have
\begin{align*}
-q^{-2}\omega_{i,1}\omega_{j,1}\omega_{k,1}\Phi_3 \Phi_1|_{q=-1}= -2ijk,
\\
q^{-5-\frac{1}{2}} \omega_{i,2}\omega_{j,2}\omega_{k,2}\Phi_5\Phi_4 \Phi_3 \Phi_1^2|_{q=-1}= 8ijk.
\end{align*}
Thus we have 
\begin{align*}
\frac{J_{M_{i,j,k}}-1}{\Phi_2}&=-2ijk+8ijk=6ijk \quad (\mod  \Phi_2),
\end{align*}
which implies the assertion.
\end{proof}

\begin{acknowledgments}
This work was partially supported by JSPS Research Fellowships for Young Scientists.
The author is deeply grateful to Professor Kazuo Habiro, Professor Tomotada Ohtsuki, and Professor Toshie Takata
for helpful advice and encouragement.
In particular she is grateful to  Professor Kazuo Habiro  for many important discussions.
She would like also to thank  Professor Naoya Enomoto for discussions concerning the content of Section \ref{Ex},
 Professor Paul Melvin for  comments  concerning the content of Remark \ref{melv}, Professor Stefan Friedl for suggestions concerning  the pictures in Figure 1.\end{acknowledgments}


\begin{thebibliography}{99}
%\bibitem{An} A. Beliakova, A. Wehrli, Categorification of the colored Jones polynomial and Rasmussen invariant of links. Canad. J. Math. \textbf{60} (2008), no. 6, 1240--1266. 
\bibitem{Bing} R. H. Bing, A homeomorphism between the 3-sphere and the sum of two solid horned
spheres, Ann. of Math. (2) \textbf{56} (1952) 354--362.
\bibitem{Ci} D. Cimasoni, Slicing Bing doubles. 
Algebr. Geom. Topol. \textbf{6} (2006), 2395--2415. 
\bibitem{Ch1} J. C. Cha,  Link concordance, homology cobordism, and Hirzebruch-type defects from iterated p-covers. J . Eur. Math. Soc. (JEMS) \textbf{12} (2010), no. 3, 555--610. 
\bibitem{Ch2} J. C. Cha, T. Kim,  Covering link calculus and iterated Bing doubles. Geom. Topol. \textbf{12} (2008), no. 4, 2173--2201.  
\bibitem{Ch3} J. C. Cha, C. Livingston, D. Ruberman,  Algebraic and Heegaard-Floer invariants of knots with slice Bing doubles. 
Math. Proc. Cambridge Philos. Soc. \textbf{144} (2008), no. 2, 403--410. 
\bibitem{Co1} {T.  D. Cochran, } Derivatives of links: Milnor's concordance invariants and Massey's products. Mem. Amer. Math. Soc. \textbf{84} (1990), no. 427, x+73 pp.
\bibitem{Co} T. D.  Cochran, S. Harvey, C. Leidy,  Link concordance and generalized doubling operators.
Algebr. Geom. Topol. \textbf{8} (2008), no. 3, 1593--1646. 
%\bibitem{Fr3} M. Freedman, The topology of four-dimensional manifolds.
%J. Differential Geom. \textbf{17} (1982), no. 3, 357--453. 
%\bibitem{Fr1} M. Freedman, The disk theorem for four-dimensional manifolds. Proceedings of the International Congress of Mathematicians, Vol. 1, 2 (Warsaw, 1983), 647--663, PWN, Warsaw, 1984. 
\bibitem{mel} {T. D. Cochran, P. Melvin},  The Milnor degree of a 3-manifold.
J. Topol. \textbf{3} (2010), no. 2, 405--423. 
\bibitem{mel2} {T. D. Cochran, P. Melvin}, Quantum cyclotomic orders of 3-manifolds. 
Topology \textbf{40} (2001), no. 1, 95--125.
\bibitem{Fr2} M. Freedman, X. S. Lin,  On the (A,B)-slice problem. Topology \textbf{28} (1989), no. 1, 91--110. 
 \bibitem{H} S. Harvey, Homology cobordism invariants and the Cochran-Orr-Teichner filtration of the link concordance group. 
Geom. Topol. \textbf{12} (2008), no. 1, 387--430. 
 \bibitem{H2}
{K. Habiro},
{A unified Witten-Reshetikhin-Turaev invariants for integral homology spheres}.
{Invent. Math. \textbf{171} (2008), no. 1, 1--81. }
\bibitem{H3} K. Habiro, Claspers and finite type invariants of links. Geom. Topol. \textbf{4} (2000), 1--83. 
\bibitem{KM} R. Kirby, P. Melvin,
The 3-manifold invariants of Witten and Reshetikhin-Turaev for sl(2,C). 
Invent. Math. \textbf{105} (1991), no. 3, 473--545. 
%\bibitem{L1} [In preparation] C. Livingston, C. Van Cott, Concordance of Bing doubles and boundary genus. Math. Proc. Cambridge Philos. Soc. \textbf{151} (2011), no. 3, 459--470. 
%\bibitem{L}
%{G. Lusztig},
%{Introduction to quantum groups}.
%{Progress in Mathematics, \textbf{110}. Birkh\"auser Boston, Inc., Boston, MA, 1993.} 
%\bibitem{Ma} C. Manolescu, B. Owens, A concordance invariant from the Floer homology of double branched covers.
%Int. Math. Res. Not. IMRN 2007, no. 20, Art. ID rnm077, 21 pp. 
\bibitem{MS} H. R. Morton, P. Strickland,  Jones polynomial invariants for knots and satellites.  Math. Proc. Cambridge Philos. Soc. \textbf{109} (1991), no. 1, 83--103.
%\bibitem{O} P. Ozv\'ath, Z. Szab\'o,  Knot Floer homology and the four-ball genus. Geom. Topol. \textbf{7} (2003), 615--639. 
\bibitem{O} T. Ohtsuki, A polynomial invariant of integral homology $3$-spheres. Math. Proc. Cambridge Philos. Soc. \textbf{117} (1995), no. 1, 83--112.
\bibitem{P} M. Petkov\v{s}ek, H. S.  Wilf, D. Zeilberger,  A=B. With a foreword by Donald E. Knuth. With a separately available computer disk. A K Peters, Ltd., Wellesley, MA, 1996. xii+212 pp.
%\bibitem{R}  J. Rasmussen, Khovanov homology and the slice genus. Invent. Math. \textbf{182} (2010), no. 2, 419--447. 
\bibitem{Re}
{N. Y. Reshetikhin, V. G. Turaev},
{Ribbon graphs and their invariants derived from quantum groups}.
{Comm. Math. Phys. \textbf{127} (1990), no. 1, 1--26.}
\bibitem{Re2}
{N. Y. Reshetikhin, V. G. Turaev},
Invariants of 3-manifolds via link polynomials and quantum groups. 
Invent. Math. \textbf{103} (1991), no. 3, 547--597. 
\bibitem{sakie1}
{S. Suzuki},
{On the universal $sl_2$ invariant of ribbon bottom tangles.}
{Algebr. Geom. Topol.  \textbf{10} (2010), no. 2, 1027--1061. }
\bibitem{sakie2}
{S. Suzuki},
{On the universal $sl_2$ invariant of boundary bottom tangles.}
{Algebr. Geom. Topol.  \textbf{12} (2012), 997--1057. }
\bibitem{sakie3}
{S. Suzuki},
{On the universal $sl_2$ invariant of Brunnian bottom tangles},
{Math. Proc. Camb. Phil. Soc.,  \textbf{154} (2013), no.1, 127--143.}
\bibitem{sakie4}
{S. Suzuki},
{On the colored Jones polynomials of ribbon links, boundary links, and Brunnian links. }
{to appear in Banach Center Publ. }
\bibitem{takata}
T. Takata,
On the set of the logarithm of the LMO invariant for integral homology $3$-spheres.
Math. Proc. Cambridge Philos. Soc. \textbf{145} (2008), no. 2, 349--361. 

\end{thebibliography}
\end{document}